\documentclass[a4paper,12pt]{amsart}
  \usepackage{amssymb,amsthm}
\usepackage{graphicx,color}        
 \usepackage{tikz}
\usetikzlibrary{automata,positioning,arrows}

\usepackage{url} 

\newtheorem{theorem}{Theorem}[section]
\newtheorem{corollary}[theorem]{Corollary}
\newtheorem{proposition}[theorem]{Proposition}
\newtheorem{lemma}[theorem]{Lemma}

\theoremstyle{definition}
\newtheorem{definition}[theorem]{Definition}

\newtheorem{remark}[theorem]{Remark}

\def\r{\mathbb R}

\begin{document}

\title[]{A relation between cylindrical critical points of Willmore-type energies, weighted areas and vertical potential energies}
\author{Rafael L\'opez}
\address{Departamento de Geometr\'{\i}a y Topolog\'{\i}a\\  Universidad de Granada.  18071 Granada, Spain}
\email{rcamino@ugr.es}
\author{\'Alvaro P\'ampano }
\address{Department of Mathematics and Statistics, Texas Tech Universityy, Lubbock, TX 79409 USA}
\email{alvaro.pampano@ttu.edu}
%
%

\begin{abstract}This paper considers energies associated with three different physical scenarios and obtains relations between them in a particular case. The first family of energies consists of the Willmore-type energies involving the integral of powers of the mean curvature which extends the Willmore and Helfrich energies. The second family of energies is the area functionals arising in weighted manifolds, following the theory developed by Gromov, when the density is a power of the height function. The third one is the free energies of a fluid deposited in a horizontal hyperplane when the potentials depend on the height with respect to this hyperplane. We find relations between each of them when the critical point is a hypersurface of cylindrical type. We prove that the generating curves coincide for the Willmore-type energies without area constraint and for weighted areas, and the similar result holds for the generating curves of Willmore-type energies and of the vertical potential energies, after suitable choices of the physical parameters. In all the cases, generating curves are critical points for a family of energies extending the classical bending energy. In the final section of the paper, we analyze the stability of a liquid drop deposited on a horizontal hyperplane with vertical potential energies. It is proven that if the free interface of the fluid is a graph on this hyperplane, then the hypersurface is stable in the sense that it is a local minimizer of the energy. In fact, we prove that the hypersurface is a global minimizer in the class of all graphs with the same boundary.
\end{abstract}

\maketitle

\section{Introduction and objectives}

This paper considers three families of energy functionals that appear in different physical scenarios, namely, Willmore-type energies, weighted areas whose densities depend on the height function and the free energies of a fluid bulk deposited on a horizontal plane under the action of some vertical potentials.

\subsection{Willmore-type energies}
 
The Willmore-type energies have their origin in the works of Germain and Poisson on the elastic theory of surfaces (\cite{ger,poi}), which was motivated by previous investigations of Euler and the Bernoulli family about one-dimensional models on the vibration of elastic beams. In order to model the bending energy of a thin elastic plate $\Sigma$, Germain proposed to consider energies which are the integral of even and symmetric functions of the principal curvatures of the surface. The simplest non-trivial choices lead to the Willmore energy
$$W[\Sigma]=\int_\Sigma H^2\,dA,$$
where $H$ is the mean curvature of $\Sigma$, a term coined by Germain (\cite{Germain}), and $dA$ is the area element. Although it was previously studied by Blaschke (\cite{B}) among others, the Willmore energy $W$ was named after the famous conjecture proposed by Willmore in 1965 (\cite{wi}). This conjecture was proven in 2014 by Marques and Neves (\cite{mn}).

Later on, in 1973, a remarkable energy was proposed by Helfrich to model lipid bilayers (\cite{He}). A thin bilayer may be modeled by a mathematical surface $\Sigma$ and, in this setting, the Helfrich energy of the surface is defined by
$$H_{a,b,c}[\Sigma]=\int_\Sigma\left(a\left(2H+c\right)^2+bK\right)dA,$$
where $K$ is the Gaussian curvature of the surface, $a,b$ are rigidity constants and $c\in\mathbb{R}$ is the so-called spontaneous curvature. Note that due to the Gauss-Bonnet theorem the total Gaussian curvature term in above energies $H_{a,b,c}$ is just a topological invariant which does not affect the associated Euler-Lagrange equation on the interior of the surface. 

More generally, for a surface $\Sigma$, Willmore-type energies are defined as  
$$W_{p,\mu,\varsigma}[\Sigma]=\int_\Sigma \left(\left(H-\frac{\mu}{2}\right)^p+\varsigma\right)dA,$$
where $p, \mu, \varsigma\in\mathbb{R}$ are constants. The Willmore-type energies $W_{p,\mu,\varsigma}$ extend the Willmore energy $W= W_{2,0,0}$ (\cite{wi}), the constrained Willmore energy $W_{2,0,\varsigma}$ (\cite{ku}), as well as the Helfrich energy $W_{2,\mu,0}$ (\cite{He}), but for the topological invariant term. Regardless of the boundary conditions, critical points of these energies $W_{p,\mu,\varsigma}$ satisfy a fourth order partial differential equation which can be expressed in terms of the mean and Gaussian curvatures of the surfaces as
\begin{equation}\label{eq-willmore}
p\Delta\left(H-\frac{\mu}{2}\right)^{p-1}+2p\left(H-\frac{\mu}{2}\right)^{p-1}\left(2H^2-K\right)-4H\left(\left(H-\frac{\mu}{2}\right)^{p}+\varsigma\right)=0,
\end{equation}
where $\Delta$ is the Laplacian on $\Sigma$ (\cite{to}). Beyond their mathematical interest, Willmore-type energies have been employed in biology to model bilipid membranes in the Canham-Helfrich-Evans models (\cite{Can,Ev,He}) as well as $\beta$-barrels (\cite{Dal,TZA}), to mention a couple.

\subsection{Weighted areas}
  
The second family of energies is motivated by the catenary, which is the curve that describes the shape of a hanging chain. The two-dimensional analogue resembles the shape of a piece of cloth hanged by its own weight. If $(x,y,z)$ denote the Cartesian coordinates of $\mathbb{R}^3$, the piece of cloth is modeled by a surface $\Sigma$, assumed to be included in the upper half-space $z>0$, which is a critical point of the energy functional
$$A_z[\Sigma]=\int_\Sigma z\,dA,$$
among all surfaces with the same boundary and the same area. The energy $A_z$ represents the action of constant gravity on the piece of cloth, measured with respect to the plane $z=0$. Following Dierkes (\cite{bht,di1,di2,dh}), critical points of $A_z$ are called singular minimal surfaces. These surfaces have the lowest center of gravity among all surfaces with the same initial data, extending the known property of the catenary. For this reason, these surfaces are used as models of perfect domes by architects such as the German Frei Otto (\cite{ot}); see also \cite{lo2}. 

One can extend the energy $A_z$ and define the weighted area energies
$$F_{\alpha,\varpi}[\Sigma]=\int_\Sigma z^\alpha\,dA+\varpi\int_\Omega z^\alpha\,dV,$$
where $\alpha,\varpi\in\mathbb{R}$ are constants, $\Omega$ is the $3$-domain between $\Sigma$ and the plane $z=0$, and $dV$ is the volume element on $\r^3$. A critical point of the energies $F_{\alpha,\varpi}$ satisfies the equation
\begin{equation}\label{eq-weighted}
2H=\alpha\frac{\nu _3}{z}+\varpi,
\end{equation}
where $\nu_3$ is the vertical component of the unit normal vector to the surface $\Sigma$. The case $\alpha=0$ corresponds with the well known surfaces of constant mean curvature (minimal surfaces if $\varpi=0$). 

\subsection{Vertical potential energies}
  
The third family of energy functionals considers the free energies of a fluid bulk $\Omega$ deposited in a horizontal plane when this fluid is affected by potentials depending on the height with respect to the supporting plane. Assuming that the vertical potential energies are powers of the height $z$ with respect to the reference plane $z=0$ and under ideal conditions of constant density and incompressibility of the fluid, the free energies of the system are the vertical potential energies 
$$E_{\eta,m,\lambda}[\Sigma]=\int_\Sigma \,dA+\eta\int_\Omega z^m\,dV+\lambda\int_\Omega\,dV,$$
where $\eta, m, \lambda\in\mathbb{R}$ are constants. The first integral is the area of $\Sigma$ (usually modeling the liquid-air interface) and it measures the surface tension. The second term is the vertical potential energy acting on the fluid bulk $\Omega$. The constant $\eta$ represents a physical quantity involving the difference between the mass density across $\Sigma$. The constant $\lambda$ is a Lagrange multiplier which reflects the physical hypothesis that the enclosed volume is fixed through any variation of the fluid. Critical points of the energies $E_{\eta,m,\lambda}$ are solutions of the equation 
\begin{equation}\label{ELE}
2H=\eta z^m+\lambda.
\end{equation}
A surface satisfying \eqref{ELE} is said to be a stationary surface. If $\eta=0$ or $m=0$, the potential energy depending on the height is neglected and $\Sigma$ is a surface with constant mean curvature (minimal surface if $\lambda=0$). If $m=1$, the physical situation is that of a liquid drop deposited on a horizontal plane in the presence of constant gravity (\cite{fi}).

\subsection{Objectives} 

Critical points for the energy functionals $W_{p,\mu,\varsigma}$, $F_{\alpha,\varpi}$ and $E_{\eta,m,\lambda}$ are unrelated in general,  unless they are constant mean curvature surfaces. Moreover, the physical scenarios from which they arise are nothing alike. Clearly, the corresponding Euler-Lagrange equations \eqref{eq-willmore}, \eqref{eq-weighted} and \eqref{ELE} are not even similar. Indeed, \eqref{eq-willmore} is an equation of order four, while \eqref{eq-weighted} and \eqref{ELE} are of second order.

The aim of this paper is to establish relations in the nontrivial case (i.e., for surfaces with nonconstant mean curvature) when the critical points of these energies have a particular geometry. The geometric property is that the critical points are cylindrical surfaces. A cylindrical surface is a surface obtained by moving a straight line parallel to itself along a curve contained in an orthogonal plane, called the generating curve. Under this assumption, the Euler-Lagrange equations \eqref{eq-willmore}, \eqref{eq-weighted} and \eqref{ELE} reduce to ordinary differential equations that must be satisfied by the generating curves. We will show that these equations are related, for suitable choices of the physical parameters: see Sections \ref{sec4} and \ref{sec5}, respectively. See  Section \ref{sec3} for a summary of these relations. What is more, we will also show that in all the cases the differential equations satisfied by the generating curves are the Euler-Lagrange equations associated with a family of functionals acting on planar curves and involving their curvatures. These functionals for curves give rise to generalized elastic curves: see Section \ref{sec2}.

Critical points of the energies $W_{p,\mu,\varsigma}$, $F_{\alpha,\varpi}$ and $E_{\eta,m,\lambda}$ are defined as those surfaces where the first variation of the energy vanishes among all admissible variations of the surface. It is natural to ask if these surfaces are also minimizers of the energy because these ones reflect the physical situation that the critical point is realizable. A weaker and necessary condition is the non-negativity of the second derivative of the energy, i.e., the stability of a given critical point. The final part of this paper investigates this problem. The stability problem for the Willmore-type energies $W_{p,\mu,\varsigma}$ is a difficult task due to the order of the equation \eqref{eq-willmore} which makes the expression of the second derivative very difficult to handle (\cite{GTT,to}). Some references of the stability of the Willmore energy are \cite{ku,pal1,pal2,ue}. The situation is similar for the weighted area functionals $F_{\alpha,\varpi}$ and it has received little attention in the literature. Recently the first author has investigated the analogous Plateau-Rayleigh phenomenon for singular minimal surfaces ($\alpha=1$ and $\varpi=0$). See \cite{lo3}.

Our main interest is to study the stability problem for the family of energies $E_{\eta,m,\lambda}$. The case $m=1$ is of great interest in the theory of capillarity and the literature is great. Here we refer to \cite{fi} for a background on the problem; see also \cite{kp,we}. For a general value for $m$, it is interesting to provide sufficient conditions so that the surface is stable. If the surface is a graph on a horizontal plane, it will be proven in Section \ref{sec6} that it is stable and, what is more, it is a global minimizer in the class of all graphs with the same boundary.

\section{Generalized elastic curves}\label{sec2}

Throughout this paper we will focus on cylindrical surfaces critical for the energies described above. These surfaces are determined by their generating curves, which will satisfy suitable ordinary differential equations. It turns out that solutions of these equations give rise to a generalization of the classical elastic curves (\cite{eu}). In this section we introduce these curves as well as their origin and respective energy functionals.

The $p$-elastic curves are critical points of a family of classical variational problems involving the curvature $\kappa$ of the curves $\gamma$. In 1738, in a letter to Euler, Daniel Bernoulli proposed to investigate extremals of the functionals
$$\mathbf{\Theta}_p[\gamma]=\int_\gamma \kappa^p\,ds,$$
where $p\in\r$ and $ds$ is the length element. For particular choices of $p$, the functionals $\mathbf{\Theta}_p$ and their critical curves are well understood nowadays. For instance, if $p=2$, $\mathbf{\Theta}_2$ represents the classical bending energy whose critical points, subject to length constraint, are the Euler-Bernoulli elastic curves (\cite{eu}). Formulated in 1691 by Jacob Bernoulli, this was the first case of the functionals $\mathbf{\Theta}_p$ considered in the literature.

Chronologically, the case $p=0$ was the second to be analyzed. The corresponding variational problem was stated in 1697 as a public challenge from Johan Bernoulli to his brother Jacob. The functional $\mathbf{\Theta}_0$ is nothing but the length functional, whose equilibria are geodesics.  
Another classical case corresponds with the functional $\mathbf{\Theta}_1$, which measures the total curvature of the curve. For planar curves, its associated Euler-Lagrange equation is an identity.

Other interesting choices for $p$ were studied in the decade of 1920 by Blaschke (\cite{B}). He showed that critical planar curves for $\mathbf{\Theta}_{1/2}$ are catenaries. Blaschke also considered the functional $\mathbf{\Theta}_{1/3}$, which represents the equi-affine length for convex curves, and proved that the critical curves are parabolas.

Although the functionals $\mathbf{\Theta}_p$ were introduced long time ago, many of their properties have yet to be fully discovered and exploited. This has motivated a great number of works and applications involving these functionals and their associated features (\cite{fkn,sw,wa}). For instance, the study of the evolution of closed planar curves under the gradient flow of $\mathbf{\Theta}_p$ has been recently investigated in \cite{np,opw,po}. Similarly, extensions of these functionals have also been considered in order to understand other topics in the areas of differential geometry and mathematical physics (\cite{AGP,LP1,LP2,LP3}). 

Throughout this paper we will consider the \emph{elastic-type} energies  
$$\mathbf{\Theta}_{p,\mu,\sigma}[\gamma]=\int_\gamma\left(\left(\kappa-\mu\right)^p+\sigma\right)ds,$$
where $\gamma$ is a planar curve and $p,\mu,\sigma\in\r$ are constants. The constant $\sigma$ is a Lagrange multiplier encoding the conservation of the length through the variation. In particular, if $\sigma=0$ there is no length constraint. The distinction between the cases $\sigma=0$ and $\sigma\neq 0$ will be essential in the discussion below. The unconstrained case ($\sigma=0$) was used in \cite[Ths. 5.2, 5.4]{LP1} to characterize the profile curves of surfaces of revolution in $\r^3$ which satisfy the linear Weingarten relation 
\begin{equation}\label{linear}
\kappa_1=a\kappa_2+b,
\end{equation}
between the principal curvatures $\kappa_1$ and $\kappa_2$  of the surface. The particular functional $\mathbf{\Theta}_{1/2,\mu,0}$ was previously defined in \cite{AGP} and used in \cite[Prop. 4.2, Th. 4.3]{AGP} to characterize invariant constant mean curvature surfaces in Riemannian and Lorentzian $3$-space forms. On the other hand, critical points of the constrained case were geometrically described in \cite{LP3}. In the same paper, the different shapes of all critical curves were shown.

Regardless of the boundary conditions, a critical curve for $\mathbf{\Theta}_{p,\mu,\sigma}$ satisfies the associated Euler-Lagrange equation
\begin{equation}\label{ELcurves}
p\frac{d^2}{ds^2}\left(\left(\kappa-\mu\right)^{p-1}\right)+p\kappa^2\left(\kappa-\mu\right)^{p-1}-\kappa\left(\left(\kappa-\mu\right)^p+\sigma\right)=0.
\end{equation}
Solutions of \eqref{ELcurves} will appear many times throughout the paper. For simplicity, we introduce the following terminology.

\begin{definition}\label{d-21} A planar curve $\gamma$ whose curvature $\kappa$ is a solution of \eqref{ELcurves} is called a \emph{generalized elastic curve}. If there is no length constraint ($\sigma=0$), then $\gamma$ is said to be a \emph{free generalized elastic curve}.
\end{definition}

The role of generalized elastic curves will be crucial in our proofs. Indeed, we will relate the critical points of any of the energies described in the introduction to those of another one passing through this concept. 

\section{Definitions and summary of results}\label{sec3}

This section is devoted to defining the energy functionals that will be studied in the paper, as well as stating the corresponding Euler-Lagrange equations. In the last part, we will summarize the relations between critical points of these energies, assuming cylindrical geometry. Since all the definitions and arguments involved also hold for arbitrary dimension of the ambient space, we will consider from now on hypersurfaces in the Euclidean space $\r^{n+1}$ and we will state the results in this context.

Let $\mathbb{R}^{n+1}$ be the Euclidean space of dimension $n+1$ ($n\geq 1$) with Cartesian coordinates $(x_1,\ldots,x_{n+1})$ and let $\r_{+}^{n+1}$ be the upper half-space $x_{n+1}>0$. Let $\Sigma$ be an oriented hypersurface of $\r^{n+1}$ and denote by $\nu:\Sigma\rightarrow\mathbb{S}^n\subset\mathbb{R}^{n+1}$ its Gauss map. The map $\nu$ will be identified with the (globally defined) unit normal vector field along $\Sigma$.  When the hypersurface $\Sigma$ is closed, $\Omega$ will represent the enclosed domain in $\mathbb{R}^{n+1}$. If $\Sigma$ is not closed, we will assume that $\Omega$ is the domain in $\mathbb{R}^{n+1}$ between $\Sigma$ and the hyperplane $x_{n+1}=0$. 

\subsection{Willmore-type energies}

For a hypersurface $\Sigma$ of $\r^{n+1}$, we define the {\it Willmore-type} energies
$$W_{p,\mu,\varsigma}[\Sigma]=\int_\Sigma\left(\left(H-\frac{\mu}{n}\right)^p+\varsigma\right)dA,$$
where $dA$ is the area element in $\r^{n+1}$ and $p,\mu,\varsigma\in\r$ are constants. The constant $\mu$ plays the role of the spontaneous curvature in the Helfrich energy, while $\varsigma$ is a Lagrange multiplier. In the case $\varsigma\neq 0$, the area of the hypersurface $\Sigma$ is preserved through the variation. Using standard methods from calculus of variations and employing compactly supported variations, we can compute the associated Euler-Lagrange equation, obtaining
\begin{equation}\label{ELWillmore}
p\Delta\left(H-\frac{\mu}{n}\right)^{p-1}+p\lvert A\rvert^2\left(H-\frac{\mu}{n}\right)^{p-1}-n^2 H \left(\left(H-\frac{\mu}{n}\right)^p+\varsigma\right)=0,
\end{equation}
where $\lvert A\rvert^2$ is the square of the norm of the second fundamental form. It is easy to check that \eqref{ELWillmore} reduces to \eqref{eq-willmore} in the two-dimensional case ($n=2$) because $\lvert A\rvert^2=4H^2-2K$ holds from the definition of the Gaussian and mean curvatures.

Following the terminology introduced in Section \ref{sec2}, we introduce the corresponding names for solutions of \eqref{ELWillmore}.

\begin{definition}\label{d-31} A hypersurface $\Sigma$ of $\r^{n+1}$ whose mean curvature $H$ satisfies \eqref{ELWillmore} is called a \emph{generalized Willmore hypersurface}. If there is no area constraint ($\varsigma=0$), then $\Sigma$ is said to be a \emph{free generalized Willmore hypersurface}.
\end{definition}

\subsection{Weighted areas}

The {\it weighted area} energies of a hypersurface $\Sigma$ are defined by
$$F_{\alpha,\varpi}[\Sigma]=\int_\Sigma x_{n+1}^\alpha\,dA+\varpi\int_\Omega x_{n+1}^\alpha\,dV,$$
where $dV$ is the volume element in $\mathbb{R}^{n+1}$, and $\alpha, \varpi\in\r$ are constants. If $\alpha\notin \mathbb{N}$, $x_{n+1}^\alpha$ is only well defined for $x_{n+1}>0$ and hence, we will assume that the hypersurface $\Sigma$ is included in $\r^{n+1}_+$. The Euler-Lagrange equation associated with $F_{\alpha,\varpi}$ can be described in terms of its mean curvature, namely, 
\begin{equation}\label{ELF}
H=\frac{\alpha}{n}\frac{\nu_{n+1}}{x_{n+1}}+\frac{\varpi}{n},
\end{equation}
where $\nu=(\nu_1,\ldots,\nu_{n+1})$ is the unit normal along $\Sigma$. Note that if $\alpha=0$, then critical points are hypersurfaces with constant mean curvature, which will not be considered here. From now on, we will assume $\alpha\neq 0$.

\begin{definition} A hypersurface $\Sigma$ of $\r^{n+1}$ (restricted to $\r^{n+1}_+$ when necessary) whose mean curvature $H$ satisfies \eqref{ELF} is called a \emph{generalized singular minimal hypersurface}.
\end{definition}

Generalized singular minimal hypersurfaces also appear in the theory of weighted manifolds developed by Gromov (\cite{gr,mo3}). We briefly describe this here. Let $\psi$ be a density function on $\r^{n+1}$ and consider the weighted area and weighted volume elements on $\r^{n+1}$ defined by $dA_\psi= e^{\psi} dA$ and $dV_\psi = e^{\psi}  dV$, respectively. Critical points of the weighted area for all weighted volume preserving variations are characterized by the equation
$n H= \langle \overline{\nabla}\psi,\nu\rangle+\varpi$, where $\overline{\nabla}$ is the gradient on $\r^{n+1}$. If we take the particular density $\psi(x_1,\ldots,x_{n+1})=\alpha\log(x_{n+1})$, $x_{n+1}>0$, then the weighted area coincides with the first term in $F_{\alpha,\varpi}$ while the weighted volume is the second integral. The function 
$$H_\psi=H-\frac{\alpha}{n}\frac{\nu_{n+1}}{x_{n+1}}$$ 
is called the weighted mean curvature of $\Sigma$. Thus the critical points of $F_{\alpha,\varpi}$ are those hypersurfaces with constant weighted mean curvature $H_\psi$.

\subsection{Vertical potential energies}

The {\it vertical potential energies} for a hypersurface $\Sigma$ are defined by
$$E_{\eta,m,\lambda}[\Sigma]=\int_\Sigma \,dA+\eta\int_\Omega x_{n+1}^m\,dV+\lambda\int_\Omega\,dV,$$
where $\eta,m,\lambda\in\mathbb{R}$ are constants. As above, if $m\notin \mathbb{N}$, we will again assume that $\Sigma$ is included in $\r^{n+1}_+$. The associated Euler-Lagrange equation for this family of functionals is
\begin{equation}\label{ELE2}
nH=\eta x_{n+1}^m+\lambda.
\end{equation}
Observe that if $\eta=0$ or $m=0$, solutions of \eqref{ELE2} are hypersurfaces with constant mean curvature. Throughout this paper we will consider $\eta\neq 0$ and $m\neq 0$.

\begin{definition} A hypersurface $\Sigma$ of $\r^{n+1}$ (restricted to $\r^{n+1}_+$ when necessary) whose mean curvature $H$ satisfies \eqref{ELE2} is called a \emph{stationary hypersurface}.
\end{definition}

\subsection{Summary of results}

Once the three families of energies have been defined, we present a summary of the results which will be proven in this paper. These results show several relations between the critical points with nonconstant mean curvature of above energies, assuming cylindrical geometry and for suitable choices of the physical parameters. For the sake of clarity and brevity, the dependence on these parameters will be omitted in the summary (for details, the reader should look at the corresponding statements).

A hypersurface $\Sigma$ of $\r^{n+1}$ generated by moving an $(n-1)$-dimensional affine space parallel to itself along a curve contained in an orthogonal plane is a \emph{cylindrical hypersurface}. By definition, $\Sigma$ is invariant in the direction of $n-1$ unit vectors $w_i\in\mathbb{R}^{n+1}$, $1\leq i\leq n-1$ and we may consider the parameterization of $\Sigma$,
\begin{equation}\label{param}
\phi(s,t)=\gamma(s)+\sum_{i=1}^{n-1} t_iw_i,\quad t=(t_1,\ldots,t_{n-1}),
\end{equation}
where $\gamma$ is the curve contained in an orthogonal plane to the linear subspace generated by all vectors $w_i$. The curve $\gamma$ is called the \emph{generating curve} of  $\Sigma$. If $n=1$, the hypersurface $\Sigma$ is nothing but the curve $\gamma$. Without loss of generality, the parameter $s\in I\subset\mathbb{R}$ denotes the arc length parameter of $\gamma$ and $\left(\,\right)'$ the derivative with respect to $s$. Let  $T(s)=\gamma'(s)$ be the unit tangent vector field along the planar curve $\gamma$ and define the unit normal vector field $N(s)$ along $\gamma(s)$ to be the counter-clockwise rotation of $T(s)$ through an angle $\pi/2$ in the plane where $\gamma$ lies. In this setting, the Frenet equation
$$T'(s)=\kappa(s)N(s),$$
defines the curvature $\kappa$ of $\gamma$. In terms of the parameterization \eqref{param}, the unit normal $\nu$ of $\Sigma$ is $\nu(s,t)=N(s)$ and the mean curvature is
\begin{equation}\label{H}
H(s,t)=\frac{\kappa(s)}{n}.
\end{equation}

Observe that if  $\Sigma$ is the curve $\gamma$ ($n=1$), then the Willmore-like energies $W_{p,\mu,\varsigma}$ and their associated Euler-Lagrange equations \eqref{ELWillmore} are nothing but the elastic-type energies $\mathbf{\Theta}_{p,\mu,\sigma}$ and \eqref{ELcurves}, respectively. This follows directly since from \eqref{H} the curvature $\kappa$ of the curve is the same as its mean curvature $H$. For arbitrary dimension, we have the following result.

\begin{proposition}\label{relationprop} Let $\Sigma$ be a cylindrical hypersurface of $\r^{n+1}$ and let $\gamma$ be its generating curve. Then $\Sigma$ is a generalized Willmore hypersurface if and only if $\gamma$ is a generalized elastic curve for the relation $\sigma=n^p\varsigma$. In particular, $\Sigma$ is a free generalized Willmore hypersurface if and only if $\gamma$ is a free generalized elastic curve.
\end{proposition}
\begin{proof} The proof follows by comparing the Euler-Lagrange equations \eqref{ELcurves} and \eqref{ELWillmore}. Using the expression of the mean curvature of a cylindrical hypersurface \eqref{H}, the Euler-Lagrange equation \eqref{ELWillmore} becomes
$$p\frac{d^2}{ds^2}\left((\kappa-\mu)^{p-1}\right)+p\lvert A\rvert^2\left(\kappa-\mu\right)^{p-1}-\kappa\left((\kappa-\mu)^p+n^p\varsigma\right)=0\,.$$
Finally, from the definition of a cylindrical hypersurface it follows that $\lvert A\rvert^2=\kappa^2$ and so \eqref{ELWillmore} is, precisely, \eqref{ELcurves} for $\sigma=n^p\varsigma$.
\end{proof}

Since for the proofs of the relations all arguments go through the concept of generalized elastic curve, Proposition \ref{relationprop} becomes essential. The rest of the relations are listed below (see also the flow diagram in Figure \ref{fig1}):
\begin{enumerate} 
\item A cylindrical hypersurface is a generalized singular minimal hypersurface if and only if it is a free generalized Willmore hypersurface  (Theorem \ref{t1}).
\item A cylindrical hypersurface is a stationary hypersurface if and only if it is a generalized Willmore hypersurface (Theorem \ref{tnew}).
\item Any generalized singular minimal cylindrical hypersurface is a stationary hypersurface (Theorem \ref{t3}).
\end{enumerate}

\begin{figure}[ht] \begin{center}
\begin{tikzpicture}[auto]

    \node[draw,rectangle,rounded corners,node distance = 1cm,
    text width=8em,  text centered,       minimum height=1cm] at (-2,0) (block1){Free generalized Willmore hypersurfaces $W_{p,\mu,0}$};

    \node[draw,  rectangle,rounded corners,
    text width=8em,    text centered,     minimum height=1cm] at (3,0) (block2){Generalized Willmore hypersurfaces $W_{p,\mu,\varsigma}$};
 
    \node[draw,rectangle,rounded corners,
    text width=8em,   text centered,     minimum height=1cm] at (-2,-2.5) (block3){Free generalized elastic curves\\ $\Theta_{p,\mu,0}$};

    \node[draw,  rectangle,rounded corners,
    text width=8em,     text centered,    minimum height=1cm] at (3,-2.5) (block4){Generalized elastic curves\\ $\Theta_{p,\mu,\sigma}$};
        
\node[rectangle, draw, 
    text width=10em, text centered, rounded corners, minimum height=4em] at (-5,-5.5)(block5) {Generalized singular minimal hypersurfaces\\ 
       $F_{\alpha\not=-1,\varpi}$};

    \node[draw,rectangle,rounded corners,
       text width=6em, text centered, 
        minimum height=1cm] at (6,-5.5) (block6){Stationary\\ hypersurfaces\\ $E_{\eta,m\not=-1,\lambda}$};
        
\draw(block1) edge[above,->] node{Def. \ref{d-31}} (block2);
  \draw(block3) edge[above,->] node{Def. \ref{d-21}} (block4);
  \draw(block2) edge[sloped, anchor=center,bend left, above,<->] node{Th. \ref{tnew}} (block6);
\draw     (block5) edge[sloped, anchor=center,above,<->] node{Th. \ref{t1}} (block3);
\draw     (block4) edge[sloped, anchor=center,above,<->] node{Th. \ref{tnew}} (block6);
\draw(block1) edge[ right,<->] node{Prop. \ref{relationprop}} (block3);
\draw(block2) edge[ right,<->] node{Prop. \ref{relationprop}} (block4);
\draw(block5) edge[above,->] node{Th. \ref{t3}} (block6);
 \draw(block1) edge[sloped, anchor=center,bend right, above,<->] node{Th. \ref{t1}} (block5);

\end{tikzpicture}
\end{center}
\caption{Relations for a cylindrical hypersurface with nonconstant mean curvature. On the arrows we indicate the results where the corresponding relations are proven.}\label{fig1}
\end{figure}
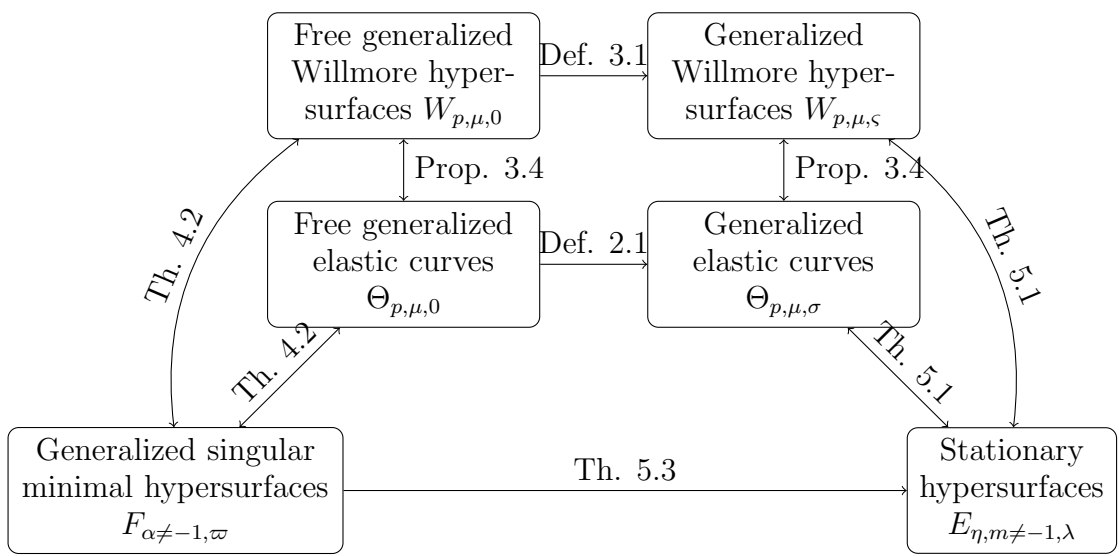

In Section \ref{sec6}, the second variation of the energies $E_{\eta,m,\lambda}$   will be analyzed asking when stationary hypersurfaces are local minimizers of the energy. The first main result is that under mild hypothesis, stationary graphs are stable. Then it will be proven that these graphs are indeed absolute minimizers for $E_{\eta,m,\lambda}$ if we compare them with other graphs with the same boundary.

\begin{theorem}\label{t-stable} Let $\Sigma$ be a compact stationary graph over the horizontal hyperplane $x_{n+1}=0$ for $m>0$. If $\eta\nu_{n+1}>0$ holds on $\Sigma$, then $\Sigma$ is stable. Moreover, $\Sigma$ is a minimizer of the energy $E_{\eta,m,\lambda}$ in the class of all graphs over $x_{n+1}=0$ and with the same boundary as $\Sigma$.
\end{theorem}

\section{Generalized singular minimal hypersurfaces}\label{sec4}

In this section we will show the relation between generalized singular minimal hypersurfaces and generalized Willmore hypersurfaces. As a first observation, note that the energies $F_{\alpha,\varpi}$ depend on two parameters and so it is reasonable to expect a relation with a biparametric family of energies instead of $W_{p,\mu,\varsigma}$. Indeed, we will see that the relation is, precisely, with the unconstrained case $W_{p,\mu,0}$.

We first prove a geometric property about the rulings of generalized singular minimal cylindrical hypersurfaces. If $\{e_1,\ldots,e_{n+1}\}$ denotes the canonical basis of $\r^{n+1}$ we will show that the rulings of any generalized singular minimal cylindrical hypersurface are all orthogonal to the vector $e_{n+1}$. 

\begin{lemma} Let $\Sigma$ be a cylindrical hypersurface with nonconstant mean curvature. If $\Sigma$ is a generalized singular minimal hypersurface, then the rulings of $\Sigma$ are orthogonal to $e_{n+1}$.
\end{lemma}
\begin{proof} A generalized singular minimal hypersurface satisfies \eqref{ELF}. In the case of a cylindrical hypersurface this equation is, using \eqref{H}, equivalent to  
\begin{equation}\label{elft}
\kappa(s)=\alpha\,\frac{\nu_{n+1}(s)}{x_{n+1}(s,t)}+\varpi.
\end{equation}
Differentiating with respect to $t_i$, for any $i=1,...,n-1$,  
$$0=-\alpha\frac{\nu_{n+1}}{x_{n+1}^2}\frac{d}{dt_i} x_{n+1}=-\alpha\frac{\nu_{n+1}}{x_{n+1}^2}\langle w_i,e_{n+1}\rangle.$$
Since $\Sigma$ has nonconstant mean curvature, then $\alpha\neq 0$ and $\nu_{n+1}\neq 0$. Thus, $\langle w_i,e_{n+1}\rangle=0$ for all $1\leq i\leq n-1$. This proves the result.
\end{proof}
 
In view of this lemma, and  after a suitable rotation around $e_{n+1}$, we may assume  that the rulings are parallel to $e_i$ for $1\leq i\leq n-1$, and that  $\gamma$ is contained in the plane spanned by $\{e_n,e_{n+1}\}$. Note that rotations about $e_{n+1}$ do not change the equilibrium condition \eqref{ELF}.  Under these assumptions, equation \eqref{elft} reads
\begin{equation}\label{elf}
\kappa(s)=\alpha\,\frac{\nu_{n+1}(s)}{x_{n+1}(s)}+\varpi.
\end{equation}
If $\kappa$ is constant, then $\gamma$ is either a straight line or a circle and the cylindrical hypersurface generated by $\gamma$ has constant mean curvature thanks to \eqref{H}. If $\kappa$ is not constant, we prove that, for suitable choices of the constants $\alpha, \varpi\in\mathbb{R}$, the equation \eqref{elf} is the Euler-Lagrange equation associated with $\mathbf{\Theta}_{p,\mu,0}$.   

\begin{theorem}\label{t1} Let $\Sigma$ be a cylindrical hypersurface with nonconstant mean curvature. Then $\Sigma$ is a generalized singular minimal hypersurface for $\alpha\neq -1$ if and only if its generating curve is a free generalized elastic curve where $p=\alpha/(\alpha+1)$ and $\mu=\varpi/(\alpha+1)$. Consequently, $\Sigma$ is a generalized singular minimal hypersurface for $\alpha\neq -1$ if and only if it is a free generalized Willmore hypersurface for the above values of $p$ and $\mu$.
\end{theorem}
\begin{proof} 
Consider an energy functional acting on planar curves of the type
$$\mathbf{\Theta}[\gamma]=\int_\gamma P(\kappa)\,ds,$$
where $P$ is a smooth function defined on an adequate domain. Then a planar curve $\gamma$ with nonconstant curvature satisfies the Euler-Lagrange equation associated with $\mathbf{\Theta}$ if and only if there is a coordinate system such that $\gamma$ can be locally expressed as $\gamma(s)=\left(\gamma_1(s),\gamma_2(s)\right)$ with
$$\gamma_2(s)=\frac{1}{\sqrt{d}}\dot{P}\left(\kappa(s)\right),$$
for any constant $d>0$ (\cite[Prop. 3.3]{LP3}). The upper dot denotes the derivative with respect to $\kappa$.  

Assume that $\gamma$ is the generating curve of a cylindrical hypersurface satisfying \eqref{ELF}. In particular, equation \eqref{elf} holds for the generating curve $\gamma$. Since $\kappa$ is not constant, from the inverse function theorem we may assume that the arc length parameter $s$ is, locally, a function of the curvature $\kappa$. Then $\gamma_2(s)=\dot{P}(\kappa)/\sqrt{d}$ for a suitable function $\dot{P}$ and constant $d>0$. Using   that $s$ is the arc length parameter, the curve $\gamma$ can be parameterized as
$$\gamma(s)=\frac{1}{\sqrt{d}}\left(-\int\left(\kappa\dot{P}-P\right)ds,\dot{P}\right).$$
It follows that $x_{n+1}(s)=\dot{P}/\sqrt{d}$ and $\nu_{n+1}(s)=-\left(\kappa\dot{P}-P\right)/\sqrt{d}$. Therefore, equation \eqref{elf} reduces to the first order ordinary differential equation in $P=P(\kappa)$,
\begin{equation}\label{ODE}
\left(\alpha+1\right)\kappa\dot{P}-\varpi\dot{P}=\alpha P.
\end{equation}
Since $\alpha\neq -1$, this equation can be integrated obtaining
$$P(\kappa)=\left(\kappa-\frac{\varpi}{\alpha+1}\right)^{\frac{\alpha}{\alpha+1}},$$
up to a multiplicative constant. Consequently, \eqref{ELcurves} is satisfied for the values of $p$ and $\mu$ of the statement and for $\sigma=0$, that is, $\gamma$ is a free generalized elastic curve.

For the converse, assume that the generating curve $\gamma$ of a cylindrical hypersurface $\Sigma$ satisfies \eqref{ELcurves}. From the above parameterization of critical curves of general curvature depending energies $\mathbf{\Theta}$, it is clear that \eqref{elf} holds and so does \eqref{ELF}. Then, $\Sigma$ is a generalized singular minimal hypersurface.

The second assertion follows from Proposition \ref{relationprop}.
\end{proof}

\begin{remark} Theorem \ref{t1} provides a relation between $F_{\alpha,\varpi}$ and $W_{p,\mu,0}$ within the class of cylindrical hypersurfaces. However, in the particular case $p=2$ (respectively, $\alpha=-2$) and $n=2$, the same relation holds for surfaces under less restrictive assumptions. In \cite[Prop. 4.1]{PP}, it was proven that any surface satisfying \eqref{ELF} for $\alpha=-2$ also satisfies the Euler-Lagrange equation associated with $W_{2,\mu,0}$, while the converse holds for disc type surfaces of revolution (\cite[Th. 4.1]{PP}).
\end{remark}

Planar curves satisfying \eqref{elf} were described in \cite{lo}. Theorem \ref{t1} is particularly illustrative in the one-dimensional case ($n=1$).

\begin{corollary} A planar curve $\gamma$ with nonconstant curvature is a generalized singular minimal curve for $\alpha\neq-1$ if and only if it is a free generalized elastic curve for $p=\alpha/(\alpha+1)$ and $\mu=\varpi/(\alpha+1)$.
\end{corollary}

This corollary explains why the catenary, which satisfies \eqref{ELF} for $\alpha=1$ and $\varpi=0$ in the classical sense of a hanging chain, also satisfies the Euler-Lagrange equation \eqref{ELcurves} of the functional $\mathbf{\Theta}_{1/2,0,0}$ studied by Blaschke (\cite{B}). Similarly, it also shows why the parabola satisfies \eqref{ELF} for $\alpha=1/2$ and $\varpi=0$ as well as \eqref{ELcurves} for $p=1/3$ and $\mu=\sigma=0$ (\cite{B}).

In the following remark we describe the analogue relation of Theorem \ref{t1} for the case $\alpha=-1$.

\begin{remark}\label{t2} Let $\Sigma$ be  a cylindrical hypersurface with nonconstant mean curvature. Then $\Sigma$ is a generalized singular minimal hypersurface for $\alpha=-1$ if and only if its generating curve satisfies the Euler-Lagrange equation associated with the curvature energy 
$$\widetilde{\mathbf{\Theta}}_{\mu}[\gamma]=\int_\gamma e^{\,\mu\kappa}\,ds,$$ 
where $\mu=1/\varpi$. Consequently, $\Sigma$ is a generalized singular minimal hypersurface for $\alpha=-1$ if and only if it satisfies the Euler-Lagrange equation associated with
$$\widetilde{W}_\mu[\Sigma]=\int_\Sigma e^{n\mu H}\,dA,$$
for $\mu=1/\varpi$.
\end{remark}

A further relation can be deduced from Theorem \ref{t1} and Remark \ref{t2}.  As it was mentioned in Section \ref{sec2}, the critical points for compactly supported variations of the energies $\mathbf{\Theta}_{p,\mu,0}$ and $\widetilde{\mathbf{\Theta}}_\mu$ characterize the profile curves of the surfaces of revolution in $\r^3$ which satisfy the linear Weingarten relation \eqref{linear} (\cite{LP1,LP2}).   Consequently, we have the following relation.

\begin{corollary} For every $n\geq 2$, there exists a one-to-one correspondence between surfaces of revolution in $\r^3$ satisfying
$$\kappa_1+\alpha\kappa_2=\varpi,$$
and generalized singular minimal cylindrical hypersurfaces of $\r^{n+1}$.
\end{corollary}
\begin{proof}
Let $n\geq 2$ be fixed. From Theorem \ref{t1} and Remark \ref{t2}, a cylindrical hypersurface of $\r^{n+1}$ satisfies \eqref{ELF} if and only if its generating curve satisfies the Euler-Lagrange equation associated with $\mathbf{\Theta}_{p,\mu,0}$ or $\widetilde{\mathbf{\Theta}}_\mu$, depending on the value of $\alpha$, for suitable choices of $p$ and $\mu$.

Similarly, it follows from \cite[Ths. 5.2, 5.4]{LP1} (and \cite[Ths. 2.1, 2.7]{LP2}) that a planar curve $\gamma$ satisfies the Euler-Lagrange equation for $\mathbf{\Theta}_{p,\mu,0}$ or $\widetilde{\mathbf{\Theta}}_\mu$ if and only if the surface of revolution generated by rotating the curve $\gamma(s)=(x(s),0,z(s))\subset\mathbb{R}^3$ around the $z$-axis satisfies a linear relation between the principal curvatures. This gives the correspondence.
\end{proof}

\section{Stationary hypersurfaces}\label{sec5}

In this section we will relate stationary hypersurfaces to generalized Willmore hypersurfaces. Recall that stationary hypersurfaces are the solutions of the Euler-Lagrange equation associated with the vertical potential energies $E_{\eta,m,\lambda}$, while generalized Willmore hypersurfaces are the solutions of the Euler-Lagrange equation for the Willmore-type energies $W_{p,\mu,\varsigma}$. Let $\Sigma$ be a cylindrical hypersurface of $\mathbb{R}^{n+1}$ parameterized by \eqref{param} and suppose that it is a stationary hypersurface satisfying \eqref{ELE2}. In terms of  the curvature $\kappa$ of $\gamma$, this equation is   
\begin{equation}\label{ele}
\kappa(s)=\eta x_{n+1}^m(s)+\lambda.
\end{equation}
If $\kappa$ is constant (recall that $\eta,m\neq 0$), then the function $x_{n+1}$ is constant on $\Sigma$. This proves that $\Sigma$ is a horizontal hypersurface and $\kappa=0$.   If  $\kappa$ is not constant,  we  extend the result \cite[Th. 3.4]{LP3} to prove that $\gamma$ satisfies the Euler-Lagrange equation associated with $\mathbf{\Theta}_{p,\mu,\sigma}$ for suitable energy parameters $p$ and $\mu$. We then conclude with the following result.

\begin{theorem}\label{tnew} Let $\Sigma$ be a cylindrical hypersurface with nonconstant mean curvature. Then $\Sigma$ is a stationary hypersurface for $m\neq -1$ if and only if its generating curve is a generalized elastic curve where $p=(m+1)/m$, $\mu=\lambda$, and $\sigma\in\mathbb{R}$. Consequently, $\Sigma$ is a stationary hypersurface for $m\neq -1$ if and only if $\Sigma$ is a generalized Willmore hypersurface for above values of $p$ and $\mu$ and $\varsigma=\sigma/n^p$. 
\end{theorem}
\begin{proof} 
We argue as in the proof of Theorem \ref{t1}. For the forward implication, it follows from the inverse function theorem that $x_{n+1}(s)$ is locally of the form $\dot{P}(\kappa)/\sqrt{d}$ for a suitable function $\dot{P}$ and constant $d>0$. Using this in combination with \eqref{ele}, we get that the function $P=P(\kappa)$  satisfies  the   ordinary differential equation 
$$\eta\dot{P}^m(\kappa)=d^{m/2}\left(\kappa-\lambda\right).$$
Since $m\neq -1$ (recall that if $m=0$ the hypersurface has constant mean curvature, which is not considered here), the solution of this equation is  
$$P(\kappa)=\left(\kappa-\lambda\right)^{\frac{m+1}{m}}+\sigma,$$
up to a multiplicative constant. We point out here that the energy parameter $\eta\neq 0$ as well as the constant $d>0$ only appear as part of the multiplicative constant and so, they are hidden in the constant of integration $\sigma\in\mathbb{R}$. However, since $\sigma\in\mathbb{R}$ is free, $\eta\neq 0$ and $d>0$ do not play any fundamental role in the relations between the energy parameters (see also Remark \ref{etadis}). Then $\gamma$ satisfies the Euler-Lagrange equation associated with $\mathbf{\Theta}_{p,\mu,\sigma}$ for $p=(m+1)/m$, $\mu=\lambda$, and $\sigma\in\mathbb{R}$. The converse follows from a direct computation. 

The second statement is a consequence of Proposition \ref{relationprop}.
\end{proof}

We now describe the analogous result for $m=-1$.
 
\begin{remark} Let $\Sigma$ be a cylindrical hypersurface with nonconstant mean curvature. Then $\Sigma$ is a stationary hypersurface for $m=-1$ if and only if its generating curve satisfies the Euler-Lagrange equation associated with the curvature energy
$$\widehat{\mathbf{\Theta}}_{\lambda,\sigma}[\gamma]=\int_\gamma\left(\log(\kappa-\lambda)+\sigma\right)ds,$$
where $\sigma\in\mathbb{R}$. Consequently, $\Sigma$ is a stationary hypersurface for $m=-1$ if and only if it satisfies the Euler-Lagrange equation associated with
$$\widehat{W}_{\lambda,\varsigma}[\Sigma]=\int_\Sigma \left(\log\left(H-\frac{\lambda}{n}\right)+\varsigma\right)dA,$$
for $\varsigma\in\mathbb{R}$.
\end{remark}

We next relate the generalized singular minimal hypersurfaces, which are the solutions of the Euler-Lagrange equation of the weighted area energies $F_{\alpha,\varpi}$, to the stationary hypersurfaces. As pointed out at the beginning of Section \ref{sec4}, the energies $F_{\alpha,\varpi}$ depend on two parameters while $E_{\eta,m,\lambda}$ is a three-parametric family. As a consequence, the relation will not be an equivalence in general. Indeed, if $\Sigma$ is a generalized singular minimal cylindrical hypersurface for $\alpha\neq -1$, then $\Sigma$ is a free generalized Willmore hypersurface, by Theorem \ref{t1}. In particular, free generalized Willmore hypersurfaces are also generalized Willmore hypersurfaces, but the converse is not true in general. We will then conclude, thanks to Theorem \ref{tnew}, that $\Sigma$ is a stationary hypersurface.

\begin{theorem}\label{t3} Let $\Sigma$ be a cylindrical hypersurface with nonconstant mean curvature. If $\Sigma$ is a generalized singular minimal hypersurface for $\alpha\neq -1$, then $\Sigma$ is also a stationary hypersurface for $m=-\alpha-1$ and $\lambda=\varpi/(\alpha+1)$.
\end{theorem}
\begin{proof}
Assume $\Sigma$ is a generalized singular minimal cylindrical hypersurface for $\alpha\neq -1$ with nonconstant mean curvature. From Theorem \ref{t1}, equivalently, $\Sigma$ is a free generalized Willmore hypersurface for $p=\alpha/(\alpha+1)$ and $\mu=\varpi/(\alpha+1)$. In other words, $\Sigma$ is a generalized Willmore hypersurface for $p=\alpha/(\alpha+1)$, $\mu=\varpi/(\alpha+1)$ and $\varsigma=0$.
 It follows from Theorem \ref{tnew} that $\Sigma$ is a stationary hypersurface for the values of $m$ and $\lambda$ given by
$$\frac{m+1}{m}=p=\frac{\alpha}{\alpha+1},\quad\quad\quad \lambda=\mu=\frac{\varpi}{\alpha+1}.$$
This finishes the proof.
\end{proof}

\begin{remark}\label{etadis} The converse of Theorem \ref{t3} is not true in general. Indeed, let $\Sigma$ be a stationary cylindrical hypersurface for $m\neq -1$. From the proof of Theorem \ref{tnew}, the generating curve is a generalized elastic curve where $p=(m+1)/m$, $\mu=\lambda$, and $\sigma\in\mathbb{R}$ is a constant of integration which cannot be determined only in terms of $\eta$, $m$ and $\lambda$.

Observe also that the parameter $\eta$ does not appear in the relations of Theorems \ref{tnew} and \ref{t3}. Indeed, it does not play any fundamental role in the results since it is hidden in the multiplicative constant arising to obtain $P(\kappa)$ and this multiplicative constant does not alter the corresponding Euler-Lagrange equation.
\end{remark}
  
The proof of Theorem \ref{t3} relies strongly on the variational characterization of the generating curves $\gamma$ as generalized elastic curves. This result is not only unexpected but also nontrivial because in order to obtain it, we need to go from a second order ordinary differential equation \eqref{elf} to a fourth order equation (of type \eqref{ELcurves}) and then, go back to the second order equation \eqref{ele}. However, a direct proof is expected to exist. The objective to finish this section is to show the result of Theorem \ref{t3} in a direct way. For it, we will only employ the generating curves of the cylindrical hypersurfaces.

We begin by obtaining an integral of the equation \eqref{ele}. If $\gamma$ is locally parameterized as a graph $\gamma(x)=\left(x,f(x)\right)$ for some function $f$, then the curvature of $\gamma$ is
$$\kappa(x)=\frac{{f}''(x)}{\left(1+{f'}^2(x)\right)^{3/2}}.$$
If $\gamma$ satisfies \eqref{ele}, then
$$\frac{f''}{\left(1+{f'}^2\right)^{3/2}}=\eta f^m+\lambda.$$
Multiplying  by $f'$, using that $m\not=-1$ and integrating, we have
\begin{equation}\label{integral}
-\frac{1}{\sqrt{1+{f'}^2}}=\frac{\eta}{m+1}f^{m+1}+\lambda f+c,
\end{equation}
where $c\in\r$ is a constant of integration. Consequently, $\gamma$ satisfies \eqref{ele} if and only if it can be locally parameterized as a graph for any solution of \eqref{integral}. We will use this first integral to obtain the conditions under which a generalized singular minimal cylindrical hypersurface is also stationary.

Assume that $\Sigma$ is a generalized singular minimal cylindrical hypersurface satisfying \eqref{ELF} for $\alpha\neq-1$. If the generating curve is locally parameterized by $\gamma(x)=\left(x,f(x)\right)$ for some positive function $f$, then $\gamma$ satisfies \eqref{elf}, that is,   
\begin{equation}\label{newintegral}
\kappa =\frac{\alpha}{f\sqrt{1+{f'}^2}}+\varpi
\end{equation}
because the function $\nu_{n+1}$ is 
$$\nu_{n+1}=\frac{1}{\sqrt{1+{f'}^2}}.$$
We now impose that $\gamma$ satisfies \eqref{ele}. If a solution of \eqref{newintegral} also satisfies \eqref{ele} we get
$$\frac{\alpha}{f\sqrt{1+{f'}^2}}+\varpi=\eta f^m+\lambda,$$
or, equivalently, 
$$\frac{1}{\sqrt{1+f'^2}}=\frac{\eta}{\alpha}f^{m+1}+\frac{\lambda-\varpi}{\alpha}f.$$
Observe that $\gamma$ satisfies \eqref{ele} if and only if $f$ is a solution of the first integral \eqref{integral}. Comparing \eqref{integral} with above identity, we deduce that it suffices to choose $c=0$,  $m=-\alpha-1$ and $\lambda=\varpi/(\alpha+1)$ so that a solution of \eqref{newintegral} also satisfies \eqref{ele}. This coincides with the statement of Theorem \ref{t3}.

\section{Stability analysis of stationary hypersurfaces}\label{sec6}

In this section we will investigate the stability of stationary hypersurfaces. These hypersurfaces are solutions of the Euler-Lagrange equation  \eqref{ELE2} associated with $E_{\eta,m,\lambda}$. More precisely, we will show Theorem \ref{t-stable}.  We will begin by showing that under mild hypothesis, compact stationary hypersurfaces have non-empty boundary. This is a  reasonable physical property because the case $m=1$ and $n=2$ corresponds with the model of a liquid drop resting on a horizontal plane. Since the liquid drop is supported on a plane, the air-liquid interface $\Sigma$ is a compact surface with non-empty boundary. In fact, the property that the boundary is non-empty is necessary for the existence of a liquid drop in equilibrium (\cite{kp}). This property can be generalized for many cases of the vertical potential energies $E_{\eta,m,\lambda}$. Note also that, intuitively, a realistic air-liquid interface should have no self-intersections.  

\begin{proposition} 
Assume that $m\in\mathbb{Z}$ is odd. If $\Sigma$ is a compact stationary hypersurface without self-intersections, then the boundary of $\Sigma$ is non-empty.
\end{proposition}
\begin{proof}
It is well known that the $x_k$-coordinate functions of any hypersurface $\Sigma$ of $\r^{n+1}$ obey $\Delta x_k=nH\nu_k$, $1\leq k\leq n+1$, where $\Delta$ is the Laplacian on $\Sigma$. Let $k=n+1$. If $\Sigma$ is a stationary hypersurface, $\Delta x_{n+1}=(\eta x_{n+1}^m+\lambda)\nu_{n+1}$ holds from \eqref{ELE2}. By contradiction, assume that the boundary of $\Sigma$ is empty and so $\Sigma$ is a closed hypersurface.   The  divergence theorem  implies 
\begin{equation}\label{closed}
0=\int_\Sigma \Delta x_{n+1}\,dA=\int_\Sigma \left(\eta x_{n+1}^m+\lambda\right)\nu_{n+1}\, dA=\int_\Sigma \eta x_{n+1}^m\nu_{n+1}\, dA,
\end{equation}
because in any closed hypersurface, $\int_\Sigma \nu_k\, dA=0$ for all $k$. Since $\Sigma$ has no self-intersections, the Jordan-Brower Separation Theorem asserts that $\Sigma$ determines a bounded domain $\Omega$ in $\r^{n+1}$ whose boundary is $\Sigma$. On the closure $\overline{\Omega}$ of $\Omega$, define the vector field $Z=\eta x_{n+1}^me_{n+1}=(0,\ldots,0,\eta x_{n+1}^m)$. The Euclidean divergence of $Z$ is $\mbox{Div}_{\r^{n+1}}Z=m\eta x_{n+1}^{m-1}$ and the divergence theorem and \eqref{closed} imply
\begin{equation*}
m\eta\int_\Omega x_{n+1}^{m-1}\, dV=\int_\Omega\mbox{Div}_{\r^{n+1}}Z\, dV=\int_\Sigma\langle Z,\nu\rangle\, dA=\int_\Sigma \eta x_{n+1}^m\nu_{n+1}\, dA=0.
\end{equation*}
However, $m\eta\int_\Omega x_{n+1}^{m-1}\, dV\not=0$ because $m-1\in\mathbb{Z}$ is even and $\eta\neq 0$. This contradiction proves the result.
\end{proof}

In what follows, we will analyze the stability of stationary hypersurfaces. For this, it will be necessary to have an expression of  the second variation formula of  $E_{\eta,m,\lambda}$. The first variation of  $E_{\eta,m,\lambda}$   for all compactly supported variations is
$$E_{\eta,m,\lambda}'(0)[u]=-\int_\Sigma\left(nH-(\eta x_{n+1}^m+\lambda)\right)u\,dA,$$
where $u\in\mathcal{C}_0^\infty(\Sigma)$ is the normal component of the velocity vector associated to the variation. As expected,  the Euler-Lagrange equation \eqref{ELE2} is compatible with the above expression $E_{\eta,m,\lambda}'(0)[u]$. The derivation of the formula for the second order variation is obtained from standard methods. Following for example  \cite{we}, we deduce  
$$E_{\eta,m,\lambda}''(0)[u]=-\int_\Sigma u\cdot L[u]\,dA,$$
where $L$ is the Jacobi operator defined by
\begin{equation*}\label{L}
L[u]=\Delta u+\left(\lvert A\rvert^2-m\eta x_{n+1}^{m-1}\nu_{n+1}\right)u,
\end{equation*}
and $\lvert A\rvert^2$ is the square of the norm of the second fundamental form of $\Sigma$. A stationary hypersurface $\Sigma$ is \emph{stable} if $E_{\eta,m,\lambda}''(0)[u]\geq 0$ for all $u\in\mathcal{C}^\infty_0(\Sigma)$. Since the operator $L$ is elliptic, standard theory for eigenvalues asserts that the stability of the hypersurface is equivalent to $\Sigma$ having Morse index zero.

The proof of Theorem \ref{t-stable} requires   the computation of $L[\nu_{n+1}]$.

\begin{proposition} Let $\Sigma$ be a stationary hypersurface of $\mathbb{R}^{n+1}$. Then
\begin{equation}\label{nn}
L[\nu_{n+1}]=-m\eta x_{n+1}^{m-1}.
\end{equation}
\end{proposition}
\begin{proof}
It is known that in any hypersurface $\Sigma$ of $\r^{n+1}$, the Laplacian $\Delta\nu$ of  the Gauss map $\nu$ involves the gradient of the mean curvature vector field $H$ by means of the equation
$$\Delta\nu+\lvert A\rvert^2\nu=-\nabla(nH).$$
By using  \eqref{ELE2}, we have 
$$\Delta\nu+\lvert A\rvert^2\nu=-\nabla(\eta x_{n+1}^m+\mu)=-m\eta x_{n+1}^{m-1}\left(e_{n+1}-\nu_{n+1}\nu\right).$$
Multiplying   by $e_{n+1}$ with  the Euclidean metric, we have
$$\Delta\nu_{n+1}+\lvert A\rvert^2\nu_{n+1}=-m\eta x_{n+1}^{m-1}\left(1-\nu_{n+1}^2\right),$$
from which the result follows.
\end{proof}

We now prove Theorem \ref{t-stable}. The proof is split in two sub-theorems according to the two assertions that appear in its statement. 

\begin{theorem}\label{t-stable2} Let $\Sigma$ be a compact stationary graph over the horizontal hyperplane $x_{n+1}=0$ for $m>0$. If $\eta\nu_{n+1}>0$ holds on $\Sigma$, then $\Sigma$ is stable.  
\end{theorem}

\begin{proof} 
 From the expression of $E_{\eta,m,\lambda}''(0)[u]$ and \eqref{nn}, we have 
\begin{equation*}
E_{\eta,m,\lambda}''(0)[\nu_{n+1}]=-\int_\Sigma \nu_{n+1}\cdot L[\nu_{n+1}]\,dA=m\int_\Sigma \eta \nu_{n+1}x_{n+1}^{m-1}\,dA>0,
\end{equation*}
since $\eta\nu_{n+1}>0$ and $m>0$. Now the proof follows an argument due to Fischer-Colbrie and Schoen  in the theory of stable minimal surfaces (\cite{fs}). For the sake of completeness, we give the proof here. Let  $u\in\mathcal{C}_0^\infty(\Sigma)$ be an arbitrary function. Since $\nu_{n+1}\not=0$, then $u=w\nu_{n+1}$ for a certain function $w\in\mathcal{C}_0^\infty(\Sigma)$. It follows that
$$\Delta u=\Delta\left(w\nu_{n+1}\right)=w\Delta \nu_{n+1}+\nu_{n+1}\Delta w+2\langle\nabla w,\nabla \nu_{n+1}\rangle.$$
Using this identity in the expression of the Jacobi operator $L$, we obtain
\begin{eqnarray*}
L[w\nu_{n+1}]&=&\Delta(w\nu_{n+1})+\left(\lvert A\rvert^2-m\eta x_{n+1}^{m-1}\nu_{n+1}\right)w\nu_{n+1}\\&=&w\cdot L[\nu_{n+1}]+\nu_{n+1}\Delta w+2\langle \nabla w,\nabla \nu_{n+1}\rangle\\&=&-m\eta x_{n+1}^{m-1}w+\nu_{n+1}\Delta w+2\langle \nabla w,\nabla\nu_{n+1}\rangle,
\end{eqnarray*}
where in the last equality the formula \eqref{nn} is used again. Now we insert $u=w e_{n+1}$ in the formula of the second variation of $E_{\eta,m,\lambda}$, obtaining 
\begin{eqnarray*}
E_{\eta,m,\lambda}''(0)[u]&=&-\int_\Sigma w\nu_{n+1}\cdot L[w\nu_{n+1}]\, dA=m\int_\Sigma \eta\nu_{n+1} x_{n+1}^{m-1} w^2\,dA\\&&-\int_\Sigma w\nu_{n+1}^2\Delta w\,dA-\int_\Sigma 2w\nu_{n+1}\langle\nabla w,\nabla \nu_{n+1}\rangle\,dA.
\end{eqnarray*}
The integral in the last line can be rewritten using the  divergence theorem as
\begin{eqnarray*}
\int_\Sigma w\nu_{n+1}^2\Delta w \, dA &=&-\int_\Sigma\langle\nabla (w \nu_{n+1}^2),\nabla w\rangle \, dA\\&=&-\int_\Sigma \nu_{n+1}^2\lvert\nabla w\rvert^2\, dA-2\int_\Sigma w \nu_{n+1}\langle \nabla w,\nabla \nu_{n+1} \rangle \, dA.
\end{eqnarray*}
Putting this identity in $E_{\eta,m,\lambda}''(0)[u]$, we conclude that
$$E_{\eta,m,\lambda}''(0)[u]=\int_\Sigma \left(m\eta\nu_{n+1}x_{n+1}^{m-1} w^2+\nu_{n+1}^2 \lvert\nabla w\rvert^2\right) dA \geq 0,$$
since $\eta\nu_{n+1}> 0$ and $m> 0$. This proves the result.
\end{proof}

We next prove that graphs are minimizers in the class of all graphs with the same boundary.

\begin{theorem}\label{t-stable3}  Let $\Sigma$ be a compact stationary graph over the horizontal hyperplane $x_{n+1}=0$ for $m>0$. If $\eta\nu_{n+1}>0$ holds on $\Sigma$, then $\Sigma$ is a minimizer of the energy $E_{\eta,m,\lambda}$ in the class of all graphs over $x_{n+1}=0$ and with the same boundary as $\Sigma$.
\end{theorem}
\begin{proof} 
Let $\Sigma$ be the graph $x_{n+1}=f(x_1,...,x_n)$ for a suitable function $f$ defined on a domain $U\subset\mathbb{R}^n$. Without loss of generality, we can assume that the unit normal vector to $\Sigma$ is 
$$\nu=\frac{1}{\sqrt{1+\lvert Df\rvert^2}}\left(-Df,1\right).$$
On the domain $U\times\r\subset\r^{n+1}$, define a vector field $X$ by translations of $\nu$ along the $x_{n+1}$-axis, that is, 
$$X(x_1,\ldots,x_{n+1})=\nu(x_1,\ldots,x_{n}).$$
Then
$$\mbox{Div}_{\r^{n+1}}(X)=-\mbox{Div}_{\r^n}\left(\frac{Df}{\sqrt{1+\lvert Df\rvert^2}}\right)=-nH=-\eta f^m-\lambda,$$
where the second equality is the expression of the mean curvature in non-parametric form and the last  equality is just the Euler-Lagrange equation \eqref{ELE2}.

Define on $U\times\r$, the vector field $Z=X+Y$, where $Y$ is given by  
\begin{equation}\label{yy}
Y=\left(\frac{\eta}{m+1}x_{n+1}^{m+1}+\lambda x_{n+1}\right)e_{n+1}.
\end{equation}
Notice that $m\not=-1$ since  $m> 0$. The divergence of $Y$ is 
$$\mbox{Div}_{\r^{n+1}}(Y)=\eta x_{n+1}^m+\lambda.$$
Thus 
$$\mbox{Div}_{\r^{n+1}}(Z)=\mbox{Div}_{\r^{n+1}}(X)+\eta x_{n+1}^m+\lambda=\eta\left(x_{n+1}^m-f^m\right).$$
We will prove that $\Sigma$ has less energy than any other graph with the same boundary. Let $x_{n+1}=g(x_1,...,x_n)$ be a graph $\widetilde{\Sigma}$ over $U$ with $f=g$ along $\partial\Sigma$. Let $\mathcal{O}$ be the oriented $(n+1)$-chain that determine $\Sigma\cup \widetilde{\Sigma}$ and denote by $\widetilde{\nu}$ the unit normal vector on $\widetilde{\Sigma}$ and compatible with the orientation of $\mathcal{O}$. The divergence theorem yields
\begin{eqnarray}\label{e3}
\int_\mathcal{O} \eta\left(x_{n+1}^m-f^m\right)\, dV &=&\int_\mathcal{O}\mbox{Div}_{\r^{n+1}}(Z)\, dV=\int_\Sigma\langle Z,\nu\rangle\, dA-\int_{\widetilde{\Sigma}}\langle Z,\widetilde{\nu}\rangle\, d\widetilde{A}\nonumber\\
&=&\int_\Sigma\left(1+\langle Y,\nu\rangle\right)\, dA-\int_{\widetilde{\Sigma}}\left(\langle\nu,\widetilde{\nu}\rangle+\langle Y,\widetilde{\nu}\rangle\right)d\widetilde{A}\nonumber\\
&\geq& \int_\Sigma\left(1+\langle Y,\nu\rangle\right)\, dA-\int_{\widetilde{\Sigma}}\left(1+\langle Y,\widetilde{\nu}\rangle\right)d\widetilde{A},
\end{eqnarray}
because $\langle\nu,\widetilde{\nu}\rangle\leq 1$. 

We claim that
$$\int_\Sigma \left(1+\langle Y,\nu\rangle\right)\, dA= E_{\eta,m,\lambda}[\Sigma],\quad \int_{\widetilde{\Sigma}}\left(1+\langle Y,\widetilde{\nu}\rangle\right)d\widetilde{A}=E_{\eta,m,\lambda}[\widetilde{\Sigma}].$$
We will check the claim only for $\Sigma$ since the arguments for $\widetilde{\Sigma}$ are analogous. 

Observe that the first term in the integrand, $\int_\Sigma 1\, dA$, measures the area of $\Sigma$. Consequently, comparing with the expression of the energy $E_{\eta,m,\lambda}$, it suffices to check the identity
\begin{equation}\label{vo1}
\int_\Omega \left(\eta x_{n+1}^m+\lambda\right)dV = \int_\Sigma\langle Y,\nu\rangle\,dA.
\end{equation}
Here 
$$\Omega=\{(x_1,\ldots,x_{n+1})\in\r^{n+1}:0<x_{n+1}<f(x_1,\ldots,x_n), (x_1,\ldots,x_n)\in U\}.$$
Since $\Omega\subset U\times\r$, the divergence theorem for the vector field $Y$ gives
$$\int_\Omega\mbox{Div}_{\r^{n+1}}(Y)\, dV=\int_\Omega\left(\eta x_{n+1}^m+\lambda\right)dV=\int_{\partial\Omega}\langle Y,N\rangle\,dA,$$
where $N$ is the unit normal vector on $\partial\Omega$. Notice that $\partial\Omega$ is composed by the hypersurface $\Sigma$,  the domain $U\times\{0\}$ in the hyperplane  $x_{n+1}=0$ and the vertical walls of $\Omega$
$$\{(x_1,\ldots,x_{n+1})\in\r^{n+1}: (x_1,\ldots,x_n)\in\partial U, 0\leq x_{n+1}\leq f(x_1,\ldots,x_n)\}.$$
 On the vertical walls of $\Omega$, $\langle Y,N\rangle=0$ because of the orthogonality between $N$ and $e_{n+1}$. On the other hand, the vector field $Y$ vanishes on $U\times\{0\}$ since $x_{n+1}=0$ and $m>0$.  Therefore $\int_{\partial\Omega}\langle Y,N\rangle\, dA=\int_\Sigma\langle Y,\nu\rangle\, dA$, proving \eqref{vo1}.  

Once the claim is proven, the inequality \eqref{e3} becomes
\begin{equation}\label{vo2}
\int_\mathcal{O} \eta\left(x_{n+1}^m-f^m\right)dV\geq E_{\eta,m,\lambda}[\Sigma]- E_{\eta,m,\lambda}[\widetilde{\Sigma}].
\end{equation}
Finally, Theorem \ref{t-stable2} is proven if we show that the left hand-side of \eqref{vo2} is nonpositive. Since $\eta>0$ and $m>0$, it is enough to prove that $x_{n+1}-f\leq 0$ holds in $\mathcal{O}$.   The chain $\mathcal{O}$ has different components where, from the divergence theorem, $dV=dx_1\ldots dx_{n+1}$ if $\nu$ points outwards of the component and $dV=-dx_1\ldots d x_{n+1}$ otherwise. In the first case, $f\geq x_{n+1}$ in the component, while in the second case $f\leq x_{n+1}$. This finishes the proof. 
\end{proof}
 
\begin{remark} The arguments in Theorem \ref{t-stable3} are based in calibrations of the theory of minimal hypersurfaces (\cite{mo1,mo2}). Since $\Sigma$ is a graph,  it is possible to define the differential $n$-form 
$$\omega(w_1,\ldots,w_{n})=\mbox{det}(w_1,\ldots,w_n,\nu)$$ 
in the domain $U\times\r$. This differential form is not closed, but it satisfies the following three conditions:
\begin{enumerate}
\item  The differential $d\omega$ is 
\begin{equation*}
d\omega=\mbox{div}_{\r^{n+1}}\left(X\right) dx_1\wedge\ldots\wedge dx_{n+1}=-\left(\eta f^m+\lambda\right) dx_1\wedge\ldots\wedge dx_{n+1}.
\end{equation*}
\item $\omega(w_1,\ldots,w_n) \leq 1$, in the set of all $n$-dimensional orthonormal frames\\ $\{w_1,\ldots,w_n\}$ of the tangent space of $U\times\r$. 
\item $\omega(w_1,\ldots,w_n)=1$ for any positive orthonormal basis on the tangent space of $\Sigma$. 
\end{enumerate}
After a modification of $\omega$ by means of the vector field $Z$, we conclude  the result of Theorem \ref{t-stable3}.
\end{remark}

\section{Concluding remarks}

The main goal of this paper was to establish relations between the critical points of the Willmore-type energies $W_{p,\mu,\varsigma}$, the weighted area functionals $F_{\alpha,\varpi}$, and the vertical potential energies $E_{\eta,m,\lambda}$. We focused on those hypersurfaces with nonconstant mean curvature of cylindrical type. These hypersurfaces have the peculiarity that any statement defined regarding them is carried out to the generating curve. For example, the Euler-Lagrange equations of the three families of energies reduce to ordinary differential equations, which correspond with the Euler-Lagrange equations of energy functionals, acting on planar curves, involving powers of the curvature of the curve. The objective has been achieved, proving an equivalence between free generalized Willmore hypersurfaces and generalized singular minimal hypersurfaces. It has also been shown that stationary hypersurfaces are in a one-to-one relation with generalized Willmore hypersurfaces. As a consequence, we concluded that any generalized singular minimal hypersurface is also a stationary hypersurface. 
 
A relevant question to investigate was whether these hypersurfaces are minimizers of the energies or not. Critical points that are also minimizers correspond with realistic equilibria in physics. In the last part of the paper we have investigated the stability of stationary graphs on $x_{n+1}=0$. If the critical point is a cylindrical hypersurface, the property of being a graph on $x_{n+1}=0$ is the same as saying that $\gamma$ is a graph on $x_{n+1}=0$.  Thus $\gamma$ can be subdivided in pieces all of which are graphs on $x_{n+1}=0$. Moreover, the component $\nu_{n+1}$ of the unit normal to $\Sigma$ coincides with the one of its generating curves. These pieces of $\gamma$ determine sub-cylindrical hypersurfaces of the initial hypersurface that are graphs on strips of $x_{n+1}=0$. Thus, the stability condition in Theorem \ref{t-stable} reduces to discerning if $\eta\nu_{n+1}$ is positive on each of the pieces of $\gamma$. A description of the generating curves of cylindrical stationary hypersurfaces appears in \cite{lo0} for the case $m=1$ and in \cite{LP3} for arbitrary $m$.

In a future work, it would be interesting to study a similar problem for other weighted energies. A family of densities with geometric interest are those related with the flow by the mean curvature. The investigation of the flow by the mean curvature is a topic of great activity in geometric analysis since the works of Huisken and Ilmanen (\cite{hu,il}). In this theory, the solitons by the flow are characterized in terms of their mean curvature. As critical points of a variational problem, they have zero weighted mean curvature. Among these solitons, the self-shrinkers and expanders deserve to be highlighted. It would be interesting to examine if the generating curves of cylindrical hypersurfaces are critical points for some energy functionals depending on the curvature of the curve, as is the case for the energies studied in this paper. If such a relation was to exist, it would show once again the important role of the theory of elastic curves initiated by Euler and Bernoulli in unrelated areas of mathematics.

\section*{Acknowledgements}
Rafael  L\'opez  is a member of the Institute of Mathematics  of the University of Granada. This work  has been partially supported by  the Projects  I+D+i PID2020-117868GB-I00, supported by MCIN/AEI/10.13039/501100011033/,  A-FQM-139-UGR18 and P18-FR-4049. \'Alvaro P\'ampano has been partially supported by the AMS-Simons Travel Grants Program 2021-2022. He would like to thank the Department of Geometry and Topology of University of Granada for its warm hospitality. The authors would also like to thank the referee for carefully reviewing the paper.

\end{document}